\documentclass[a4paper,20pt,eqno]{article}

\usepackage[latin1]{inputenc}
\usepackage[english]{babel}
\usepackage[]{amsthm}
\usepackage[]{amsmath}
\usepackage[T1]{fontenc}
\usepackage[]{amssymb}
\usepackage[]{mathtools}
\usepackage[]{enumitem}
\usepackage[]{xcolor}

\theoremstyle{remark}
\newtheorem{rmk}{Remark}[section]
\theoremstyle{plain}

\newtheorem{lem}{Lemma}[section]
\newtheorem{thm}{Theorem}[section]
\newtheorem{prop}{Proposition}[section]

\theoremstyle{definition}
\newtheorem{defin}{Definition}[section]

\newtheorem*{ack}{Acknowledgements}

\newcommand{\F}{\mathcal F}

\newcommand{\Oo}{\mathcal O}
\newcommand{\N}{\mathbb N}
\newcommand{\R}{\mathbb R}
\newcommand{\C}{\mathbb C}

\newcommand{\E}{\mathcal E}

\newcommand{\X}{ X^*}
\newcommand{\Ker}{\mathrm{Ker}}

\def\ds{\displaystyle}
\newcommand{\one}{1\mkern -4mu\mathrm{l}}

\DeclarePairedDelimiter{\prodscal}{\langle}{\rangle}

\title{On the Ornstein-Uhlenbeck operator in convex sets of Banach spaces}
\author{G. Cappa}

\begin{document}
\maketitle

\begin{abstract}
We study the Ornstein-Uhlenbeck operator and the Ornstein-Uhlenbeck semigroup in an open convex subset of an infinite dimensional separable Banach space $X$. This is done by finite dimensional approximation. In particular we prove Logarithmic-Sobolev and Poincaré inequalities, and thanks to these inequalities we deduce spectral properties of the Ornstein-Uhlenbeck operator.
\end{abstract}

\textbf{2010 Mathematics Subject Classification AMS} 35R15, 39B62, 47D07


\section*{Introduction}
In this paper we describe the properties of the Ornstein-Uhlenbeck operator $L^\Omega$, and of the semigroup generated by it $(T^\Omega(t))_{t\geq0}$, in $L^2(\Omega,\gamma)$. Here $\Omega$ is an open convex subset of an infinite dimensional separable Banach space $X$ endowed with a centered non-degenerate Gaussian measure. We define the Ornstein-Uhlenbeck operator as the self-adjoint associated to the quadratic form (see e.g. \cite{RocknerMa}, \cite{Wang})
\[\int_\Omega\langle \nabla_Hu,\nabla_Hv\rangle_H \gamma(dx),\quad u,v\in W^{1,2}(\Omega,\gamma)\]
where $W^{1,2}(\Omega,\gamma)$ is the Sobolev space defined in Section \ref{section:construc_OU}, and $\nabla_H$ is the gradient along the Cameron-Martin space $H$.

We approximate  $L^\Omega$ by finite-dimensional Ornstein-Uhlenbeck operators, by using the cylindrical approximation of $\Omega$ made in \cite{LunMirPal}. For finite dimensional Ornstein-Uhlenbeck operators we use the results of \cite{BerFor} and some properties that we prove here. In particular we show that
\[[T^{\Omega}(t)(fg)]^2\leq T^{\Omega}(t)(f^2)T^{\Omega}(t)(g^2),\text{ a.e. in }\Omega,\ \forall f,g\in L^2(\Omega,\gamma),\ \forall t\geq0,\]
and
\[|\nabla_HT^{\Omega}(t)(f)|_H\leq e^{-t}T^{\Omega}(t)|\nabla_Hf|_H,\text{ a.e. in }\Omega,\ \forall f\in W^{1,2}(\Omega,\gamma),\ \forall t\geq0.\]

Moreover we prove that $(T^\Omega(t))_{t\geq0}$ is a submarkovian semigroup, namely $0\leq f(x)\leq 1$, $\gamma$-a.e. implies $0\leq (T^\Omega(t)f)(x)\leq 1$, $\gamma$-a.e. for every $t>0$.
These properties are used to show the \emph{Poincaré inequality}
 \[\int_\Omega\left|f-\int_\Omega f\,d\gamma\right|^2d\gamma\leq\int_\Omega|\nabla_Hf|_H^2d\gamma,\]
and the \emph{Logarithmic-Sobolev inequality}
 \[\int_\Omega f^2\log(f^2)d\gamma\leq\int_{\Omega}|\nabla_H f|_H^2d\gamma+\|f\|_{L^2(\Omega,\gamma)}^2\log(\|f\|_{L^2(\Omega,\gamma)}^2),\]
that hold for every $f\in W^{1,2}(\Omega,\gamma)$.

Such inequalities can also be deduced from the theorems shown in \cite[Section 6]{FeyelUstunel} where the proofs make heavy use of Malliavin calculus and Stochastic Analysis. Our proof is much simpler and relies on analytic tools and on the Deuschel-Strook's method.
Infinite dimensional Poincaré and Logarithmic-Sobolev inequalities are proved in \cite{bogachev} for $\Omega=X$ through the Wiener chaos decomposition. In our case we don't have an explicit representation formula for the semigroup neither any sort of Wiener chaos decomposition or explicit expression of the eigenfunctions.

As expected, thanks to the Poincaré inequality we prove spectral properties of $L^\Omega$. 

\section{Construction of Ornstein-Uhlenbeck operator and some properties of semigroups}\label{section:construc_OU}
Let $\X_\gamma$ be the closure of $\X$ in $L^2(X,\gamma)$ and let $R_\gamma:\X_\gamma\rightarrow(\X)'$ be the covariance operator of $\gamma$ defined by
\[R_\gamma f (g):=\int_X f(x)g(x)\,\gamma(dx),\quad f\in\X_\gamma,\ g\in\X.\]
The range of $R_\gamma$ is contained in $X$, that is for every $f\in\X_\gamma$ there exists a unique $y\in X$ such that $R_\gamma f(g)=g(y)$ for all $g\in\X$ (see \cite[p.~44]{bogachev}).
The Cameron-Martin space $H$ is given by $R_\gamma(\X_\gamma)$, i.e., it consists of all $h\in X$ such that there exists $\hat h\in \X_\gamma$ for which
\[h=R_\gamma\hat h.\]
The Cameron-Martin space $H$ is a separable Hilbert space equipped with the inner product
\[\prodscal{h,k}_H=\int_X R_\gamma\hat h(x) R_\gamma\hat k(x)\,\gamma(dx),\]
and the norm $|h|^2_H:=\prodscal{h,h}_H$ for $h,k\in H$, see \cite[p.~60]{bogachev}. Let $\{h_n\}_{n\in\N}$ be an orthonormal basis of $H$. For $f:X\rightarrow\R$ we denote by $\partial_n f$ the directional derivative in the direction of $h_n$
\[\partial_n f(x)=\lim_{t\rightarrow0}\frac{f(x+t h_n)-f(x)}{t},\]
whenever such limit exists.

We recall the integration by parts formula
\begin{equation}
\label{parti}
\int_X \partial_j f\,\varphi\, d\gamma = - \int_X f( \partial_j \varphi - \hat{h}_j)d\gamma,
\end{equation}
that holds for every $f$, $\varphi\in C^1_b(X)$ (e.g.,  \cite[Thm.~5.1.8]{bogachev}).

In finite dimension we denote by $D_i$ the directional derivative in the direction of the $i-$th vector of the canonical basis of $\R^n$.

\begin{defin}
 $\F C_b^k(X)$ is the space of cylindrical functions of the form
 \[f(x)=\varphi(l_1(x),\ldots,l_n(x)),\]
 with $\varphi\in C^k_b(\R^n)$, $l_i$, $\ldots$, $l_n\in X^*$ and $n\in\N$.
\end{defin}

If $f\in \F C^1_b(X)$ then the Taylor expansion to the first order of $f$ at $x_0\in X$ is  \[f(x)=f(x_0)+f'(x_0)(x-x_0)+o(\|x-x_0\|_X)\]
where $f'(x_0)\in\X$. Since $H$ is continuously embedded in $X$, the map $h\mapsto f'(x_0)h$, from $H$ to $\R$, belongs to $H^*$. Then, there exists a unique $y\in H$ such that
\[\prodscal{y,h}_H=f'(x_0)h\quad \forall h\in H.\]
We put $\nabla_H f(x_0):=y$.

Let $\Omega$ be an open   set of $X$. We denote by $\F C^k_b(X)_{|\Omega}$ the set of the restrictions to $\Omega$ of the elements of $\F C^k_b(X)$.

\begin{prop}
\label{closability}
 The operator $\nabla_H:\F C^1_b(X)_{|\Omega}\rightarrow L^2(\Omega,\gamma;H)$ is closable in $L^2(\Omega,\gamma)$. The closure is still denoted by $\nabla_H$.
\end{prop}
\begin{proof}
Let $f_n\in \F C^1_b(X)$ be such that $f_{n|\Omega} \to 0$ in $L^2(\Omega,\gamma)$, and $\nabla_H f_{n|\Omega} \to G$ in $L^2(\Omega,\gamma;H)$, as $n\to \infty$. To show that $G=0$, it is sufficient to prove that for every $j\in\N$ we have
\begin{equation}
\label{zero}
\int_{\Omega} \langle G, h_j\rangle_H\varphi\, d\gamma =0,
\end{equation}
for every $\varphi\in  C^1_b(X)$ with support contained in $\Omega$. For such functions $\varphi$ and for every $n\in \N$ we have
$$\int_{\Omega} \langle \nabla f_n, h_j\rangle_H \varphi\, d\gamma = \ds \int_{X} \partial_jf_n\,  \varphi\, d\gamma$$
so that, using  \eqref{parti},
$$\int_{\Omega} \langle \nabla f_n, h_j\rangle_H \varphi\, d\gamma = -  \int_{X}f_n(  \partial_j\varphi - \hat{h}_j \varphi) d\gamma  = \ds  -\int_{\Omega} f_n(  \partial_j\varphi - \hat{h}_j \varphi) d\gamma $$
and \eqref{zero} follows letting $n\to \infty$ in both members.
\end{proof}

\begin{defin}
\label{def}
  We denote by $W^{1,2}(\Omega,\gamma)$ the domain of $\nabla_H$.
\end{defin}

The space $W^{1,2}(\Omega,\gamma)$ is a Hilbert space with the inner product
\begin{align*}
&\prodscal{f,g}_{W^{1,2}(\Omega,\gamma)}:=\int_\Omega f\,g\,d\gamma+\int_\Omega\prodscal{\nabla_H f,\nabla_H g}_H\,d\gamma,
\end{align*}
and the norm (equivalent to the graph norm)
\begin{equation}
 \label{norma}
 \|f\|^2_{W^{1,2}(\Omega,\gamma)}=\int_{\Omega}f^2d\gamma+\int_{\Omega}|\nabla_H f|^2_H d\gamma.
\end{equation}

The quadratic form
\begin{equation}
  \E(u,v)=\int_\Omega\langle \nabla_Hu,\nabla_Hv\rangle_H \gamma(dx) \quad u,v\in W^{1,2}(\Omega,\gamma)
  \label{forma_quadratica}
\end{equation}
is therefore a symmetric, closed, and coercive form, according to the notation of \cite{RocknerMa}. It is
used to define the Ornstein-Uhlenbeck operator $L^{\Omega}:D(L^{\Omega})\subset L^2(\Omega,\gamma)\rightarrow L^2(\Omega,\gamma)$ by setting
\begin{equation}
\begin{split}
D(L^{\Omega}):=\left\{\right.u\in W^{1,2}(\Omega,\gamma)&:\exists f\in L^2(\Omega,\gamma)\text{ s.t. }\\
&\left.\E(u,v)=-\int_\Omega fvd\gamma, \forall v\in W^{1,2}(\Omega,\gamma)\right\}
\end{split}
  \label{domain_L}
\end{equation}
and $L^{\Omega}u=f$. The operator $L^{\Omega}:D(L^{\Omega})\mapsto X$ is self-adjoint in $L^2(\Omega,\gamma)$ and dissipative (that is $\langle L^{\Omega}u,u\rangle_{L^2(\Omega,\gamma)}\leq0$ for every $u\in D(L^{\Omega})$), hence it is the infinitesimal generator of an analytic contraction semigroup $(T^{\Omega}(t))_{t\geq0}$ in $L^2(\Omega,\gamma)$. So $T^{\Omega}(t)f\in D(L^{\Omega})$ for every $t>0$ and
\begin{equation}
  \frac{\partial}{\partial t}T^{\Omega}(t)(f)=L^{\Omega} T^{\Omega}(t)(f)\quad\forall f\in L^2(\Omega,\gamma).
  \label{derivata_di_T(t)}
\end{equation}
See \cite[Sect. I.2]{RocknerMa}.

For every function $v:X\mapsto \R$ we set as usual $v^+(x) := \max\{v(x), 0\}$.

\begin{lem}
\label{lem:positivity_T(t)}
  The semigroup $T^\Omega(t)$ is sub-markovian, that is for all $\varphi\in L^2(\Omega,\gamma)$ such that  $\varphi\geq 0$ a.e. in $\Omega$ and   for all $t\geq0$ we  have $0\leq (T^{\Omega}(t)\varphi)(x) \leq 1$ a.e. in $\Omega$.
\end{lem}

\begin{proof} By \cite[Prop. I.4.3]{RocknerMa} it is sufficient to show that for every $u\in D(L^\Omega)$ we have
$$\int_{\Omega} L^\Omega u(u-1)^+ d\gamma \leq 0. $$
To this aim we show preliminarly  that for every $v\in W^{1,2}(\Omega, \gamma)$ the function $v^+$ belongs to $W^{1,2}(\Omega, \gamma)$.
   We set
  \[g_n(\xi)=\frac{1}{2}\bigg( \xi + \sqrt{\xi^2+\frac{1}{n}}\bigg),\quad \xi\in\R.\]
  Let $\{v_n\}\subset\F C^1_b(X)$ be a sequence that approaches $v$ in $W^{1,2}(\Omega,\gamma)$ and pointwise a.e. in $\Omega$. Then the function
  $g_n\circ v_n$ belongs to $\F C^1_b(X)$ and approaches $(v+|v|)/2 = v^+$ in $W^{1,2}(\Omega,\gamma)$.
  Indeed, $g_n\circ v_n \to v^+$ in $L^2(\Omega,\gamma)$ by the Do\-minated Convergence Theorem, and $\nabla_H(g_n\circ v_n)=(g'_n\circ v_n)\nabla_H v_n$ converges to $  \nabla_H v \one_{\{v>0\}} +   \nabla_H v \one_{\{v=0\}}/2$ in $L^2(\Omega,\gamma;H)$, still by the Dominated Convergence Theorem. So, $v^+\in W^{1,2}(\Omega,\gamma)$ and $  \nabla_Hv^+ =  \nabla_H v \one_{\{v>0\}} +   \nabla_H v \one_{\{v=0\}}/2$.

Let now $u\in D(L^\Omega)$. Then $u-1 \in W^{1,2}(\Omega, \gamma)$, so that $(u-1)^+  \in W^{1,2}(\Omega, \gamma)$,  and by definition
$$\int_{\Omega} L^\Omega u(u-1)^+ d\gamma = - \int_{\Omega}\langle \nabla_Hu, \nabla_H(u-1)^+ \rangle_H d\gamma . $$
Therefore,
$$\begin{array}{lll}
\ds \int_{\Omega} L^\Omega u(u-1)^+ d\gamma & =& \ds   -  \int_{\{u > 1\}}\langle  \nabla_Hu, \nabla_H(u-1)\rangle_H d\gamma
\\
\\
& &\ds - \frac{1}{2}  \int_{\{u= 1\}}\langle \nabla_Hu, \nabla_H(u-1)\rangle_H  d\gamma
\\
\\
 & =&  \ds   -  \int_{\{u\geq 1\}} |\nabla_Hu|^2_H d\gamma  - \frac{1}{2}  \int_{\{u= 1\}}  |\nabla_Hu|^2_H d\gamma
 \\
 \\
 &  \leq & 0.
 \end{array}$$
\end{proof}

\section{Properties in finite dimension}

Let $\gamma_d$ be the standard Gaussian measure on $\R^d$. Then, $H=\R^d$ and $|x|_H = |x|_{\R^d}$ for every $x\in \R^d$. As a canonical basis of $H$ we take the canonical basis of $\R^d$. So, $\nabla_Hf = \nabla f$ for every $f\in C^1_b(\R^d)$.

Throughout this section $\Oo$ is an open convex set in $\R^d$ with $C^{2+\alpha}$ boundary, for some $\alpha >0$.

According to Definition \ref{def},   the space $W^{1,2}(\Oo,\gamma_d)$ is the domain of the closure of the operator $\nabla: C^1_b(\R^d)_{|\Oo}\mapsto L^2( \Oo, \gamma_d; \R^d)$. Namely, it is the
set of all $f\in L^2( \Oo, \gamma_d)$ such that there exists a sequence $(f_n)\subset C^1_b(\R^d)$ such that ${f_n}_{|\Oo}\to f$ in $L^2( \Oo, \gamma_d)$, and $(D_if_n)$ converges in $L^2( \Oo, \gamma_d)$ for every $i=1, \ldots, d$. The norm \eqref{norma} is now
\[\|f\|_{W^{1,2}(\Oo,\gamma_d)}^2:=\int_\Oo |f|^2 d\gamma_d+\sum_{i=1}^d \int_\Oo |D_if|^2 d\gamma_d.  \]
Similarly, $W^{2,2}(\Oo,\gamma_d)$ is the domain of the closure of the operator $f\mapsto (\nabla f, D^2f): C^2_b(\R^d)_{|\Oo}\to L^2( \Oo, \gamma_d; \R^d)\times L^2( \Oo, \gamma_d; \R^{d^2})$ in $L^2(\Oo,\gamma_d)$, where $D^2f$ denotes the Hessian matrix of $f$ (the proof of the closability of this operator is the same as in Proposition \ref{closability}). Its norm is
\[\|f\|_{W^{2,2}(\Oo,\gamma_d)}^2:=\int_\Oo |f|^2 d\gamma_d+\sum_{i=1}^d \int_\Oo |D_if|^2 d\gamma_d+\sum_{i,j=1}^d \int_\Oo |D_i D_j f|^2 d\gamma_d.\]

\begin{lem}
\label{Le:equiv}
 The Sobolev space $W^{1,2}(\Oo,\gamma_d)$ coincides with the set of the functions $f\in W^{1,2}_{loc}(\Oo, dx)$  such that $f\in L^2(\Oo,\gamma_d)$ and whose  weak first order  derivatives are in $L^2(\Oo,\gamma_d)$. Similarly, the Sobolev space $W^{2,2}(\Oo,\gamma_d)$ coincides with the set of the functions $f\in W^{2,2}_{loc}(\Oo, dx)$  such that $f\in L^2(\Oo,\gamma_d)$ and whose  weak derivatives of order up to $2$ are in $L^2(\Oo,\gamma_d)$.
\end{lem}

\begin{proof}
  We give the proof of the second statement. Let $f\in W^{2,2}(\Oo,\gamma_d)$, $\zeta\in C_0^\infty(\R^d)$ , and let $\{f_j\}_{j\in\N}\subset C^2_b(\R^d)$ be a sequence such that $f_{j|\Oo}\rightarrow f$ in $W^{2,2}(\Oo,\gamma_d)$. Then the sequence $\{\zeta f_j\}_{j\in\N}$ converges to $\zeta f \in W^{2,2}(\Oo,dx)$, whence  $f \in W^{2,2}_{loc}(\Oo,dx)$. It is easily seen that the derivatives in the sense of $ W^{2,2}(\Oo,\gamma_d)$ can be taken for the derivatives in the sense of $W^{2,2}_{loc}(\Oo,dx)$.

The  proof of the converse is a rephrasing of Lemma 3.1 of  \cite{PratoLun}. First, we consider a compactly supported function $f\in W^{2,2}_{loc}(\Oo, dx)$  such that $f\in L^2(\Oo,\gamma_d)$ and whose  weak first and second order derivatives  are in $L^2(\Oo,\gamma_d)$. Then, $f\in W^{2,2}(\Oo, dx)$ and, since $\partial \Oo$ is $C^2$, it has an extension belonging to $W^{2,2}(\R^d, dx)$. Such extension is approximated in $W^{2,2}(\R^d, dx)$ by a sequence  of smooth functions $(f_n)$ with compact support. The restrictions of the functions $f_n$ to $\Oo$ approach $f$ in  $W^{2,2}(\Oo, dx)$ and in $W^{2,2}(\Oo, \gamma)$, therefore $f\in W^{2,2}(\Oo, \gamma)$.

\noindent If $f$ has not compact support we consider a smooth cut-off function $\theta$ such that $\theta\equiv 1$ in $B(0, 1)$, $\theta \equiv 0$ outside $B(0,2)$, and we define
$f_n(x) := f(x)\theta(x/n)$. Each $f_n$ belongs to $W^{2,2}(\Oo, \gamma)$, and using the Dominated Convergence Theorem it is easy to see that $f_n$ and its first and second order derivatives converges to $f$ and to its first and second order derivatives, respectively, in $L^{2}(\Oo, \gamma)$. Therefore, $f\in W^{2,2}(\Oo, \gamma)$.
\end{proof}

Let $\widetilde{L}:D(\widetilde{L})\rightarrow L^2(\Oo,\gamma_d)$ be defined by
\[D(\widetilde{L})=\left\{f\in W^{2,2}(\Oo,\gamma_d):\ \Delta f-\langle x,\nabla f\rangle\in L^2(\Oo,\gamma_d)\text{ and }\frac{\partial f}{\partial\nu}=0\right\}\]
where $\nu(x)$ is the exterior normal vector to $\partial\Oo$ at $x$, and
\begin{equation}
\widetilde{L}f(x)=\Delta f(x)-\langle x,\nabla f(x)\rangle\text{ for every }f\in D(\widetilde{L})\text{ and for a.e. }x\in\Oo.
\label{formula_di_L}
\end{equation}

\begin{lem}
\label{LOo_uguale_Ltilde}
  If $f\in D(\widetilde{L})$ then $f\in D(L^\Oo)$, and $L^\Oo f=\widetilde{L}f$. Therefore $L^\Oo$ is a self-adjoint extension of  $\widetilde{L}$.
\end{lem}

\begin{proof}
  Let $f\in D(\widetilde{L})$ and $\varphi\in W^{1,2}(\Oo,\gamma_d)$. Using  the integration by parts formula we get
  $$\int_\Oo \Delta f(x)\,\varphi(x)\,\gamma_d(dx)=\int_\Oo (- \prodscal{\nabla f(x),\nabla\varphi(x)} + \langle x,\nabla f(x)\rangle)\gamma_d(dx) $$
  (the boundary integral vanishes, since $\partial f/\partial \nu =0$). Therefore, 
  \[\int_\Oo \left(\Delta f(x)-\langle x,\nabla f(x)\right)\varphi(x)\gamma_d(dx)=\int_\Oo \prodscal{\nabla f(x),\nabla\varphi(x)}\gamma_d(dx).\]
  So we conclude that $f\in D(L^\Oo)$, and $L^\Oo f=\widetilde{L}f$.
\end{proof}

The operator $\widetilde{L}$ is self-adjoint, see \cite{PratoLun}. Since self-adjoint operator have not proper self-adjoint extensions, we get $\widetilde{L}=L^\Oo$.

We put
\begin{equation}
  C^{1}_\nu(\overline{\Oo})=\left\{g\in C^1_b({\overline{\Oo}}): \frac{\partial g}{\partial \nu}(x)=0,\quad x\in\partial\Oo\right\}
  \label{def_C1nu}
\end{equation}
The realization of $\Delta f-\langle \cdot,\nabla f\rangle$ in $C_b(\overline{\Oo})$ is studied in \cite{BerFor}. In particular they proved that for every $f\in C_b(\overline{\Oo})$ there exists a unique bounded solution $u(t,x)$ of problem
\begin{equation}
  \begin{dcases}
  u_t(t,x) - \Delta u(t,x)+ \prodscal{x,\nabla u(t,x)}=0,\ &t>0,\ x\in\Oo,\\
  \frac{\partial u}{\partial\nu}(t,x)=0, &t>0,\ x\in\partial\Oo,\\
  u(0,x)=f(x) &x\in\Oo,
\end{dcases}
\label{prob_Neuman_dim_finita}
\end{equation}
Setting $P_tf = u(t,\cdot)$, then $(P_t)_{t\geq0}$ is a positivity-preserving contraction semigroup in $C_b(\overline{\Oo})$. Moreover
\begin{equation}
P_t:C_b(\overline{\Oo})\rightarrow C^{1}_\nu(\overline{\Oo})\text{ for all }t>0.
\label{T(t)_manda_Cb_in_C1nu}
\end{equation}
By \cite[Proposition 4.1]{BerFor} we get
\begin{equation}
  |\nabla P_t f(x)|\leq e^{-t}P_t|\nabla f(x)|\quad \forall f\in C^{1}_\nu(\overline{\Oo}),\ \forall x\in\overline{\Oo},\ t\geq0
  \label{diseq:P_t_gradiente}
\end{equation}

The following lemma is useful for generalizing the previous estimate.

\begin{lem}
\label{integral_superficie_finito}
  Let $\Oo\subset\R^d$ be a convex therefore
  \[\int_{\partial \Oo}  G_d dH^{d-1}<\infty,\]
  where $G_d(x) = e^{-|x|^2/2}(2\pi)^{-d/2}$ and $H^{d-1}$ is the Hausdorff $(d-1)$-dimensional surface measure.
\end{lem}

\begin{proof}
  Firs we suppose that $\Oo$ is bounded then $\Oo\subset B_r$ where $B_r$ is the the ball centered in the origin with radius $r$. Therefore the projection $\pi_\Oo:\R^d\rightarrow\R^d$ is a $1-$Lipschitz function, $\pi_\Oo(\partial B_r)=\partial\Oo$ and
  \[H^{d-1}(\partial\Oo)=H^{d-1}(\pi_\Oo(\partial B_r))\leq H^{d-1}(\partial B_r)=d\ \omega_d r^{d-1}\]
  where $\omega_d$ is the surface of the $d-$dimensional unit sphere.
  Therefore
  \[p_\gamma(\Oo,\R^d):=\frac{1}{(2\pi)^{d/2}}\int_{\partial\Oo}e^{-|x|^2/2}H^{d-1}(dx)\leq\frac{d\,\omega_d}{(2\pi)^{d/2}}r^{d-1}.\]
  Then the perimeter of a convex, contained in a ball, grows polynomially with the radius of the ball.

  Now let $\Oo$ be an unbounded convex set, then
  \[p_\gamma(\Oo,R^d)=\sum_{n=0}^\infty p_\gamma(\Oo,B_{n+1}\backslash B_n).\]
  We observe that
  \[\begin{split}
    p_\gamma(\Oo,B_{n+1}\backslash B_n)&= \frac{1}{(2\pi)^{d/2}}\int_{\partial\Oo\cap(B_{n+1}\backslash B_n)} e^{-|x|^2/2}H^{d-1}(dx)\\
  &\leq\frac{e^{-n^2/2}}{(2\pi)^{d/2}}\int_{\partial\Oo\cap(B_{n+1}\backslash B_n)} H^{d-1}(dx)\\
  &\leq \frac{e^{-n^2/2}}{(2\pi)^{d/2}}\int_{\partial\Oo\cap B_{n+1}} H^{d-1}(dx),
  \end{split}\]
  moreover
  \[\int_{\partial\Oo\cap B_{n+1}} H^{d-1}(dx)\leq H^{d-1}(\partial B_{n+1})d\,\omega_d (n+1)^{d-1}.\]
  Therefore
  \[p_\gamma(\Oo,R^d)=\sum_{n=0}^\infty p_\gamma(\Oo,B_{n+1}\backslash B_n)\leq \frac{d\,\omega_d}{(2\pi)^{d/2}}\sum_{n=0}^\infty (n+1)^{d-1} e^{-n^2/2}<\infty.\]
\end{proof}

The estimate \eqref{diseq:P_t_gradiente} may be extended to Sobolev functions, as follows.
\begin{lem}
  \label{lem:stima_grad_semigrup_dim_finita}
  \begin{itemize}
    \item[(i)] $T^\Oo(t)=P_t$ on $C_b(\Oo)$.
    \item[(ii)] The space $C^{1}_\nu(\overline{\Oo})$ is dense in $W^{1,2}(\Oo,\gamma_d)$.
    \item[(iii)] For all $f\in W^{1,2}(\Oo,\gamma_d)$ we have
      \[|\nabla T^{\Oo}(t) f(x)|\leq e^{-t}T^{\Oo}(t)|\nabla f(x)|\]
      for a.e. $x\in\Oo$ and $t\geq0$.
  \end{itemize}
\end{lem}

\begin{proof}
 To prove statement (i)  we introduce the operator
$$\begin{array}{lll}
D(\mathcal A) & = &  \{ u\in C_b(\overline{\Oo}) \cap W^{2,p}(\Oo \cap B(0,R)) \; \forall p>1, \; R>0:
\\
\\
&& \;\; \Delta u -\langle \cdot, \nabla u\rangle\in C_b(\overline{\Oo}) , \;  \frac{\partial u}{\partial\nu} =0\; {\rm at }\,\partial \Oo\}
\end{array}$$
$${\mathcal A}u = \Delta u -\langle \cdot, \nabla u\rangle, \quad u\in D(\mathcal A), $$
which is the weak generator of  $P_t$ in $C_b(\overline{\Oo})$ (\cite[Proposition 3.4]{BerFor}). In particular, $P_tu\in D(\mathcal A)$ and $ {\mathcal A}P_tu = P_t({\mathcal A}u)$ for every $u\in  D(\mathcal A)$ and $t>0$.

As a first step, we show that $P_tu = T^\Oo(t)u$, for every $u\in D(\mathcal A)$.
By \cite[Proposition 3.5]{BerFor}, $D(\mathcal A) \subset C^1_b(\Oo)$. By Lemma \ref{Le:equiv}, this implies that $D(\mathcal A) \subset W^{1,2}(\Oo, \gamma_d)$.
Moreover, for every $u\in (\mathcal A)$ and $\varphi \in W^{1,2}(\Oo, \gamma_d)$ we have
\begin{equation}
\label{utile}
\int_{\Oo} \langle \nabla u, \nabla \varphi \rangle d\gamma_d = -\int_{\Oo} {\mathcal A}u\,\varphi \, d\gamma_d.
\end{equation}
Indeed, the functions $u_n:=u\theta_n$, where $\theta_n$ are the cut-off functions used in Lemma \ref{Le:equiv}, belong to $W^{2,2}(\R^d)$ and have compact support. Therefore,
$$\int_{\Oo} \langle \nabla u_n, \nabla \varphi \rangle d\gamma_d = -\int_{\Oo} {\mathcal A}u_n\,\varphi \, d\gamma_d + \int_{\partial \Oo} \frac{\partial u_n}{\partial \nu} \varphi \,G_d \,dH^{d-1}$$
where $G_d(x) = e^{-|x|^2/2}(2\pi)^{-d/2}$ and $H^{d-1}$ is the Hausdorff $(d-1)$-dimensional surface measure. Since $\nabla u_n(x) = \theta (x/n) \nabla u(x)  + \frac{1}{n}u(x)\nabla \theta(x/n) $, by the Dominated Convergence Theorem the left hand side converges to the left hand side of \eqref{utile} as $n\to \infty$. Similarly, the first integral in the right hand side converges to the right hand side of  \eqref{utile}, since $ {\mathcal A}u_n (x) =  {\mathcal A}u(x)  \theta (x/n) + u (x)(\frac{1}{n^2} \Delta \theta (x/n) - \frac{1}{n}\langle x, \nabla \theta(x/n)\rangle) -\frac{2}{n} \langle  \nabla u, \nabla \theta(x/n)\rangle $.
Since $\partial u_n/\partial \nu (x) = \frac{1}{n}u(x)\langle \nabla \theta(x/n), \nu(x) \rangle$,   the modulus of the boundary integral does not exceed
$$\frac{1}{n}\|u\|_{\infty}\|\nabla \theta\|_{\infty}  \int_{\partial \Oo}  G_d dH^{d-1}$$
(we recall that $ \int_{\partial \Oo}  G_d dH^{d-1}$ is finite, see Lemma \ref{integral_superficie_finito}), and therefore the boundary integral vanishes as $n\to \infty$.

So, \eqref{utile} holds, and it implies that $u\in D(L^\Oo)$ and $L^\Oo u =  {\mathcal A}u$. Still by \cite[Proposition 3.4]{BerFor},  $\|P_h u -u\|_{\infty}/t$ is bounded for $0
 <h<1$, and $\lim_{h\to 0}( (P_hu)(x) - u(x))/h =  {\mathcal A}u(x)$ for every $x\in \Oo$, so that by the Dominated Convergence Theorem we have also $\lim_{h\to 0} (P_hu -u)/h =  {\mathcal A}u$ in $L^2(\Oo, \gamma_d)$, and by the semigroup property,  $\lim_{h\to 0} (P_{t+h}u -P_{t}u)/h =  {\mathcal A}P_{t}u = L^\Oo P_{t}u $ in $L^2(\Oo, \gamma_d)$ for every $t\geq 0$. So, $d/dt (P_{t}u - T^\Oo(t)u) =0$ for every $t\geq 0$. Since $P_0 u = T^\Oo(0)u =u$, it follows that $P_{t}u =T^\Oo(t)u$ for every $t>0$.

Let now $f\in C_b(\overline{\Oo})$, fix any $\lambda >0$ and set $u= R(\lambda, {\mathcal A})f $. Since $L^\Oo $ is an extension of ${\mathcal A}$, we have $u= R(\lambda,L^\Oo)f$, and for every $t>0$
 $$\begin{array}{lll}
 P_t f & = & P_t(\lambda u - {\mathcal A}u) = (\lambda I- {\mathcal A})P_tu =  (\lambda I- L^\Oo ) T^\Oo(t)u = T^\Oo(t)(\lambda u - L^\Oo u)
 \\
 \\
 & = &T^\Oo(t)f, \end{array}$$
namely statement (i) holds. 

Now we prove statement (ii). If $f\in W^{1,2}(\Oo,\gamma_d)$ there exists $\{f_n\}\subset C^1_b(\overline{\Oo})$ such that $f_n \rightarrow f$ in $W^{1,2}(\Oo,\gamma_d)$ and, up to a subsequence, a.e. in $\Oo$; now $T^{\Oo}(t)f_n\rightarrow f_n$ for $t\rightarrow0^+$ in $W^{1,2}(\Oo,\gamma_d)$ and thanks to \eqref{T(t)_manda_Cb_in_C1nu} we have $T^{\Oo}(t)f_n\in C^1_\nu(\overline{\Oo})$. If we set $g_n=T^{\Oo}(\frac{1}{n})f_n$ then $g_n\in C^1_\nu(\overline{\Oo})$ for all $n\in\N$ and $g_n\rightarrow f$ in $W^{1,2}(\Oo)$.

Let us prove the third statement. For $f\in W^{1,2}(\Oo,\gamma_d)$ there exists a sequence $\{g_n\}_{n\in\N}\subset C^1_\nu(\overline{\Oo})$ such that $g_n\rightarrow f$ in $W^{1,2}(\Oo)$. Along a subsequence $\{g_{n_k}\}$ we have
\[\lim_{k\rightarrow\infty}|\nabla T^{\Oo}(t)g_{n_k}(x)|= |\nabla T^{\Oo}(t)f(x)|\quad\text{a.e. in }\Oo\]
and again, up to a subsequence
\[\lim_{k\rightarrow\infty}\left(T^{\Oo}(t)|\nabla g_{n_k}|\right)(x)= \left(T^{\Oo}(t)|\nabla f|\right)(x)\quad\text{a.e. in }\Oo.\]
By applying \eqref{diseq:P_t_gradiente} we get
\[\begin{split}
 |\nabla T^{\Oo}(t)f(x)|&=\lim_{k\rightarrow\infty}|\nabla T^{\Oo}(t)g_{n_k}(x)| \leq \lim_{k\rightarrow\infty}e^{-t} \left(T^{\Oo}(t)|\nabla g_{n_k}|\right)(x)\\
 &=e^{-t} \left(T^{\Oo}(t)|\nabla f|\right)(x)
\end{split}\]
for a.e. $x\in\Oo$.
\end{proof}

\begin{prop}
\label{prop_T(t)(fg)_dim_finita}
For all $f$, $g\in L^2(\Oo,\gamma_d)$ we have
\[\left[T^{\Oo}(t)(fg)\right]^2\leq T^{\Oo}(t)(f^2) T^{\Oo}(t)(g^2)\quad \text{a.e. in }\Oo,\text{ and }t\geq0.\]
\end{prop}

\begin{proof}
  First we consider $f,g\in C^1_\nu(\overline{\Oo})$. We fix $\varepsilon>0$ and we put
  \[v^{\varepsilon}(t,x)=\sqrt{(T^{\Oo}(t)(f^2)(x)+\varepsilon)\cdot( T^{\Oo}(t)(g^2)(x)+\varepsilon)}.\]
  Notice that $v^{\varepsilon}$ is well defined since $T^{\Oo}(t)$ is positivity-preserving by Lemma \ref{lem:positivity_T(t)}.
  For convenience we set \[\alpha=\alpha(t,x)=T^{\Oo}(t)(f^2)(x),\quad\beta=\beta(t,x)=T^{\Oo}(t)(g^2)(x).\]
  Then $v^{\varepsilon}=\sqrt{(\alpha+\varepsilon)(\beta+\varepsilon)}$, and we remark that
  \[D_i v^{\varepsilon}=D_i(\sqrt{(\alpha+\varepsilon)(\beta+\varepsilon)}) =\frac{1}{2\sqrt{(\alpha+\varepsilon)(\beta+\varepsilon)}} ((\beta+\varepsilon)D_i\alpha+(\alpha+\varepsilon)D_i\beta)\]
  moreover
  \[\begin{split}
    v^{\varepsilon}_t&=\frac{\partial v^{\varepsilon}(t,x)}{\partial t}=\frac{1}{2\sqrt{(\alpha+\varepsilon) (\beta+\varepsilon)}}\left(\frac{\partial \alpha}{\partial t}(\beta+\varepsilon)+ (\alpha+\varepsilon)\frac{\partial \beta}{\partial t}\right)\\
    &=\frac{1}{2\sqrt{(\alpha+\varepsilon)(\beta+\varepsilon)}} ((\beta+\varepsilon) L^{\Oo}\alpha+(\alpha+\varepsilon) L^{\Oo}\beta).
  \end{split}\]
  Then
  \begin{equation*}
   \begin{split}
    L^{\Oo} v^{\varepsilon}&=\sum_{i=1}^d D_{ii}v^\varepsilon-x_i D_iv^\varepsilon=\sum_{i=1}^d D_{i}(D_i v^\varepsilon)-x_i D_iv^\varepsilon\\
      &=\sum_{i=1}^d D_{i}\left(\frac{1}{2\sqrt{(\alpha+\varepsilon)(\beta+\varepsilon)}}((\beta+\varepsilon) D_i\alpha+(\alpha+\varepsilon) D_i\beta)\right)+\\
      &\phantom{=}-x_i\left(\frac{1}{2\sqrt{(\alpha+\varepsilon)(\beta+\varepsilon)}}((\beta+\varepsilon) D_i\alpha+(\alpha+\varepsilon) D_i\beta)\right)\\
      &=\sum_{i=1}^d-\frac{1}{4}\frac{1}{((\alpha+\varepsilon)(\beta+\varepsilon))^{3/2}}\left((\beta+\varepsilon) D_i\alpha+(\alpha +\varepsilon)D_i\beta\right)^2+\\
      &\phantom{=}+\frac{1}{\sqrt{(\alpha+\varepsilon)(\beta+\varepsilon)}}D_i\alpha\cdot D_i\beta+\\
      &\phantom{=}+\frac{1}{2\sqrt{(\alpha+\varepsilon)(\beta+\varepsilon)}}((\beta+\varepsilon) L^{\Oo}\alpha+(\alpha+\varepsilon) L^{\Oo}\beta)\\
      &=v^\varepsilon_t+\sum_{i=1}^d-\frac{1}{4}\frac{1}{((\alpha+\varepsilon)(\beta+\varepsilon))^{3/2}}\left[((\beta+\varepsilon) D_i\alpha)^2+((\alpha+\varepsilon) D_i\beta)^2+\right.\\
      &\phantom{=} \left.+2(\alpha+\varepsilon)(\beta+\varepsilon) D_i\alpha\cdot D_i\beta\right]+\frac{1}{\sqrt{(\alpha+\varepsilon)(\beta+\varepsilon)}}D_i\alpha\cdot D_i\beta\\
      &=v^\varepsilon_t+\sum_{i=1}^d-\frac{1}{4}\frac{1}{((\alpha+\varepsilon)(\beta+\varepsilon))^{3/2}}\left[((\beta+\varepsilon) D_i\alpha)^2+((\alpha+\varepsilon) D_i\beta)^2\right]\\
      &\phantom{=} +\frac{1}{2\sqrt{(\alpha+\varepsilon)(\beta+\varepsilon)}(\alpha+\varepsilon)(\beta+\varepsilon)}(\beta+\varepsilon) D_i\alpha\cdot(\alpha+\varepsilon) D_i\beta\\
      &=v^\varepsilon_t+\sum_{i=1}^d-\frac{1}{4}\frac{1}{((\alpha+\varepsilon)(\beta+\varepsilon))^{3/2}}\left((\beta+\varepsilon) D_i\alpha-(\alpha+\varepsilon) D_i\beta\right)^2\\
      &\leq v^\varepsilon_t.
   \end{split}
  \end{equation*}
  Therefore $v^\varepsilon_t\geq L^{\Oo}(v^\varepsilon)$, $v^\varepsilon(0,x)=\sqrt{(f(x)^2+\varepsilon)(g(x)^2+\varepsilon)}$ and $v^\varepsilon$ satisfies the Neumann boundary condition, thus for the maximum principle (\cite[Proposition 2.1]{BerFor}) we have $v^\varepsilon(t,x) \geq T^{\Oo}(t)[\sqrt{(f^2+\varepsilon)(g^2+\varepsilon)}](x)$ for all for all $x\in \overline{\Oo}$, $t\geq0$ and all $\varepsilon>0$, that is
  \[T^{\Oo}(t)(\sqrt{(f(x)^2+\varepsilon)(g(x)^2+\varepsilon)})(x)\leq\sqrt{(T^{\Oo}(t)(f^2)(x)+\varepsilon)\cdot (T^{\Oo}(t)(g^2)(x)+\varepsilon)}.\]
  Taking the $L^2$-limit as $\varepsilon\rightarrow0^+$ we obtain
  \[T^{\Oo}(t)(|fg|)(x)\leq\sqrt{T^{\Oo}(t)(f^2)(x)\cdot T^{\Oo}(t)(g^2)(x)}\]
  for all $x\in \overline{\Oo}$, $t\geq0$ and $f$, $g\in C^1_\nu(\overline{\Oo})$. Now if $f$, $g\in L^2(\Oo,\gamma_d)$ there exist $\{f_n\}$, $\{g_n\}\subset C^1_\nu(\overline{\Oo})$ such that $f_n\rightarrow f$ and $g_n\rightarrow g$ in $L^2(\Oo,\gamma_d)$. For subsequences $\{f_{n_k}\}$, $\{g_{n_k}\}$ we have
  \[T^{\Oo}(t)(|f_{n_k}g_{n_k}|)\rightarrow T^{\Oo}(t)(|fg|)\quad\text{a.e. in }\Oo\]
  and again, up to subsequences we have
  \[T^{\Oo}(t)\left(f_{n_k}^2\right)\rightarrow T^{\Oo}(t)(f^2)\quad\text{a.e. in }\Oo,\]
  \[T^{\Oo}(t)\left(g_{n_k}^2\right)\rightarrow T^{\Oo}(t)(g^2)\quad\text{a.e. in }\Oo.\]
  Therefore
  \[
  \begin{split}
   T^{\Oo}(t)(|fg|)(x)&=\lim_{k\rightarrow\infty}T^{\Oo}(t)\left(\left|f_{n_k}g_{n_k}\right|\right)(x)\\
   &\leq\lim_{k\rightarrow\infty}\sqrt{T^{\Oo}(t)\left(f_{n_k}^2\right)(x)\cdot T^{\Oo}(t)\left(g_{n_k}^2\right)(x)}\\
   &=\sqrt{T^{\Oo}(t)(f^2)(x)\cdot T^{\Oo}(t)(g^2)(x)}
  \end{split}
  \]
  for a.e. $x\in\Oo$.
\end{proof}

\section{The Ornstein-Uhlenbeck operator in infinite dimension}
The following lemma will be used several times, the proof is easy and it is left to the reader.
\begin{lem}
\label{lem:integrale_scontrato_phi_Cb}
  If $u\in L^2(\Omega,\gamma)$ and
  \begin{equation*}
    \int_{\Omega} u\varphi_{|\Omega}\ d\gamma\leq0,\quad\forall\varphi\in\F C_b(X),\ \varphi\geq0,
  \end{equation*}
  then $u\leq0$ a.e. in $\Omega$.
\end{lem}

For the next proof we will use the approximating sequence of cylindrical open convex sets $\{\Omega_n\}_n$ defined in \cite{LunMirPal}: for each $n\in\N$, $\Omega_n=\pi_n^{-1}(\Oo_n)$ (see \cite[Proposition A.5]{LunMirPal}), where $\Oo_n$ is an open smooth convex subset of a finite dimensional subspace $F_n\subset H$, of dimension $j=j(n)$, moreover $F_n\subset F_{n+1}$ for all $n\in\N$. Let $\{h_i\}_{i\in\N}$ be an orthonormal basis of the Cameron-Martin space $H$ such that
\[F_n=\mathrm{span}\{h_1,\ldots,h_{j(n)}\}.\]
The map $\pi_n:X\rightarrow F_n$ is the finite dimensional projection defined by
\[\pi_n(x)=\sum_{i=1}^{j(n)}\widehat{h}_i(x)h_i.\]
Moreover $\Omega_{n+1}\subset\Omega_n$, $\partial\Oo_n$ is smooth, $\Omega\subset\Omega_n$ and we have
\[\overline{\Omega}=\bigcap_{n\in\N}\overline{\Omega}_n,\quad\gamma(\partial\Omega)=\gamma(\partial\Omega_n)=0,\text{ and }\gamma\left(\bigcap_{n\in\N}\Omega_n\backslash\Omega\right)=0.\]

Let $L^{(n)}:D(L^{(n)})\rightarrow L^2(\Omega_n,\gamma)$ be the self-adjoint operator associated to the quadratic form \eqref{forma_quadratica} with $\Omega=\Omega_n$, and $D(L^{(n)})$ is defined by \eqref{domain_L} with $\Omega=\Omega_n$. Let $(T^{(n)}(t))_{t\geq0}$ be the semigroup generated by $L^{(n)}$ in $L^2(\Omega_n,\gamma)$.

We recall the Proposition 3.3 proved in \cite{LunMirPal}:
\begin{prop}
  \label{prop:limite_Tn_T}
  For any $f\in L^2(X,\gamma)$ and any $\lambda\in \C\backslash(-\infty,0]$,
  \[\lim_{n\rightarrow\infty}\left(R(\lambda,L^{(n)})(f_{|\Omega_n})\right)=R(\lambda,L^{\Omega})(f_{|\Omega})\quad\text{ in } W^{1,2}(\Omega,\gamma),\]
  and
  \[\lim_{n\rightarrow\infty}(T^{(n)}(t){u}_{|\Omega_n})_{|\Omega}= T^\Omega(t)u_{|\Omega}\quad\text{ in }W^{1,2}(\Omega,\gamma)\]
  for any $u\in L^2(X,\gamma)$ and $t>0$.
\end{prop}

Fix $n\in\N$, let $q\in\N$ with $q\geq j(n)=\dim F_n$  and let $G=\mathrm{span}\{h_1,\ldots,h_q\}$. Let $\pi_G:X\rightarrow G$ be the finite rank projection defined by
\[\pi_G x=\sum_{i=1}^q\widehat{h}_i(x)h_i.\]
We denote by $\gamma_G$ the induced measure $\gamma\circ \pi_G^{-1}$ in $G$. If $G$ is identified with $\R^q$ through  the isomorphism $x\mapsto(\widehat{h}_1(x),\ldots,\widehat{h}_q(x))$ for $x\in G$, then $\gamma_G$ is the standard Gaussian measure in $\R^q$. Setting $d=q-\dim F_n$, let $\Oo:=\Oo_n\times \R^{d}$ so that $\pi_G(\Omega_n)=\Oo$; we remark that $\Oo$ is an open convex subset of $G$. Let $L^G$ be the Ornstein-Uhlenbeck operator defined by \eqref{forma_quadratica} with $\Omega=\Oo$ and $\gamma=\gamma_G$
and let $T^G(t)$ be the associated semigroup.

\begin{lem}
\label{lem:identifico_TG_con_Tn}
If $w\in D(L^G)$ then the map $x\mapsto w(\pi_G(x))$ belongs to $D(L^{(n)})$ and
\[L^{(n)}(w\circ\pi_G)=(L^G w)\circ\pi_G.\]
\end{lem}

\begin{proof}
  We show that there exists $f\in L^2(\Omega_n,\gamma)$ such that, given $\varphi\in W^{1,2}(\Omega_n,\gamma)$ we have
\begin{equation}
\int_{\Omega_n}\langle\nabla_Hw(\pi_G(x)),\nabla_H\varphi(x)\rangle_H\gamma(dx)=\int_{\Omega_n}f(x)\varphi(x)\gamma(dx).
\label{eq:def_integr_operatore_OU}
\end{equation}
We remark that
\begin{equation*}
  \begin{split}
    \langle\nabla_Hw(\pi_G(x)),\nabla_H\varphi(x)\rangle_H&=\sum_{i=1}^{q}\frac{\partial w}{\partial \xi_i}(\pi_G(x))\frac{\partial\varphi}{\partial h_i}(x)\\
  \end{split}
\end{equation*}
Moreover $\varphi(x)=\varphi(\pi_G(x)+(I-\pi_G)(x))$, the space $X$ can be split as $X=G\times\widetilde{G}$ where $\widetilde{G}=(I-\pi_G)(X)$ and also $\gamma=\gamma_G\otimes\gamma_{\widetilde{G}}$ where $\gamma_{\widetilde{G}}=\gamma\circ(I-\pi_G)^{-1})$. Let $\xi\in\R^q$ and let
\[g_y(\xi):=\varphi(\xi_1 h_1+\ldots+\xi_qh_q+y),\quad y\in\widetilde{G}\]
then
\begin{equation*}
  \begin{split}
    \int_{\Omega_n}&\langle\nabla_Hw(\pi_G(x)),\nabla_H\varphi(x)\rangle_H\gamma(dx)\\
    &=\int_{\Omega_n}\sum_{i=1}^{q}\frac{\partial w}{\partial \xi_i}(\pi_G(x))\frac{\partial\varphi}{\partial h_i}(\pi_G(x)+(I-\pi_G)(x)) \gamma(dx)\\
    &=\int_{\widetilde{G}}\int_{\Oo}\sum_{i=1}^{q}\frac{\partial w}{\partial \xi_i}(\xi)\frac{\partial g_y}{\partial \xi_i}(\xi) \gamma_G(d\xi)\ \gamma_{\widetilde{G}}(dy)\\
    &=-\int_{\widetilde{G}}\int_{\Oo}L^G w(\xi)g_y(\xi)\gamma_G(d\xi)\ \gamma_{\widetilde{G}}(dy)=-\int_{\Omega_n}(L^G w)(\pi_G(x))\varphi(x)\gamma(dx).
  \end{split}
\end{equation*}
Then $w\circ\pi_G\in D(L^{(n)})$ and $L^{(n)}(w\circ\pi_G)=(L^G w)\circ\pi_G$.
\end{proof}

\begin{lem}
\label{lem:T_G_continuo}
Let $\widetilde{v}\in C^\infty_b(G)$. Then the function defined by
\[g(t)(x):=T^G(t)\widetilde{v}_{|\Oo}(\pi_G(x)),\quad t>0,\ x\in\Omega_n\]
belongs to $C((0,\infty);D(L^{(n)}))$.
\end{lem}

\begin{proof}
By Lemma \ref{lem:identifico_TG_con_Tn} it follows that $g(t)\in D(L^{(n)})$ for all $t>0$. Let us prove that $g$ is continuous at $t_0>0$. For $t>0$ we have
\begin{equation*}
  \begin{split}
    \int_{\Omega_n}|g(t)(x)&-g(t_0)(x)|^2\gamma(dx)\\
    &=\int_{\Omega_n}|T^G(t)\widetilde{v}_{|\Oo}(\pi_G(x))- T^G(t_0)\widetilde{v}_{|\Oo}(\pi_G(x))|^2\gamma(dx)\\
    &=\int_\Oo|T^G(t)\widetilde{v}_{|\Oo}(\xi)- T^G(t_0)\widetilde{v}_{|\Oo}(\xi)|^2\gamma_G(d\xi)
  \end{split}
\end{equation*}
that goes to zero as $t\rightarrow t_0$ thanks to the strong continuity of $T^G(t)$ in $L^2(\Oo,\gamma_G)$. Moreover
\begin{equation*}
  \begin{split}
    \int_{\Omega_n}|L^{(n)}g(t)(x)&-L^{(n)}g(t_0)(x)|^2\gamma(dx)\\
    &=\int_{\Omega_n}|L^{(n)}T^G(t)\widetilde{v}_{|\Oo}(\pi_G(x))- L^{(n)}T^G(t_0)\widetilde{v}_{|\Oo}(\pi_G(x))|^2\gamma(dx)\\
    &=\int_\Oo|L^{G}T^G(t)\widetilde{v}_{|\Oo}(\xi)- L^{G}T^G(t_0)\widetilde{v}_{|\Oo}(\xi)|^2\gamma_G(d\xi)
  \end{split}
\end{equation*}
that goes to zero as $t\rightarrow t_0$ since $t\mapsto T^G(t)\widetilde{v}$ belongs to $C((0,\infty),D(L^G))$. Then $g\in C((0,\infty),D(L^{(n)}))$.
\end{proof}

\begin{thm}
For all $f\in W^{1,2}(\Omega,\gamma)$ and all $t\geq0$ we have
\begin{equation}
  |\nabla_H T^{\Omega}(t)f|_H\leq e^{-t}T^{\Omega}(t)|\nabla_H f|_H \text{ a.e. on }\Omega.
  \label{eq:stima_T(t)_gradiente}
\end{equation}
\end{thm}

\begin{proof}
First we prove that \eqref{eq:stima_T(t)_gradiente} holds true for the restriction to $\Omega$ of any cylindrical regular function. Let $f\in\F C_b^\infty(X)$ and $\varphi\in\F C_b(X)$ with $\varphi\geq0$. Then
\[
 \begin{split}
  &f(x)=v(l_1(x),\ldots,l_k(x))\\
  &\varphi(x)=w(l_{k+1}(x),\ldots,l_{k+m}(x))
 \end{split}
\]
with $l_i\in X^*$, $v\in C^\infty_b(\R^k)$ and $w\in C_b(\R^m)$, $w\geq0$. We want to estimate the integral
\[\int_{\Omega_n}\varphi(x)|\nabla_H T^{(n)}(t)f_{|\Omega_n}(x)|_H\gamma(dx).\]
Let $G:=\text{span}\{F_n,R_\gamma(l_1),\ldots,R_\gamma(l_{k+m})\}$. Then $G$ is a subspace of $H$ of dimension $q\leq n+k+m$; setting $d=q-\dim F_n$ let $\Oo:=\Oo_n\times\R^d$. Let $\{\widehat{h}_i\}_{i=1}^q$ be an orthonormal basis of $G$ such that $\{\widehat{h}_1,\ldots,\widehat{h}_j\}$ is an orthonormal basis of $F_n$ if $\dim F_n=j$.

Finally we write:
\[
 \begin{split}
  &f(x)=\widetilde{v}(\pi_G(x))\\
  &\varphi(x)=\widetilde{w}(\pi_G(x))
 \end{split}
\]
where $\widetilde{v}\in C^\infty_b(G)$, $\widetilde{w}\in C_b(G)$.

Now we prove that
\begin{equation}
  (T^{(n)}(t)f_{|\Omega_n})(x)=T^G(t)\widetilde{v}_{|\Oo}(\pi_G(x)):=g(t)(x)\quad t>0,\ x\in \Omega_n.
  \label{eq:identita_Tn_TG}
\end{equation}
 We know that the function $t\mapsto T^G(t)\widetilde{v}_{|\Oo}$ belongs to $C(\left[\left. 0,\infty\right)\right.;L^2(\Oo,\gamma_G))\cap C^1((0,\infty);L^2(\Oo,\gamma_G))\cap C((0,\infty);D(L^G))$ and it is the unique classical solution of
\[\left\{
   \begin{split}
    &u'(t)=L^Gu(t),\quad t>0\\
    &u(0)=\widetilde{v}_{|\Oo}
   \end{split}
\right.
\]
in the space $L^2(\Oo,\gamma_G)$. Then for fixed $t_0\in[0,+\infty)$ we have
\begin{equation*}
  \begin{split}
    \int_{\Omega_n}|g(t)(x)&-g(t_0)(x)|^2\gamma(dx)\\
    &=\int_{\Omega_n}|T^G(t)\widetilde{v}_{|\Oo}(\pi_G(x))- T^G(t_0)\widetilde{v}_{|\Oo}(\pi_G(x))|^2\gamma(dx)\\
    &=\int_\Oo|T^G(t)\widetilde{v}_{|\Oo}(\xi)- T^G(t_0)\widetilde{v}_{|\Oo}(\xi)|^2\gamma_G(d\xi)
  \end{split}
\end{equation*}
that goes to zero for $t\rightarrow t_0$ thanks to the strong continuity of $T^G(t)$ in $L^2(\Oo,\gamma_G)$. Let us prove the differentiability of $g$ at $t_0\in(0,+\infty)$. For any $t$ such that $t+t_0>0$ we have
\begin{equation*}
  \begin{split}
    &\int_{\Omega_n}\left|\frac{g(t+t_0)(x)-g(t_0)(x)}{t}-L^GT^G(t_0)\widetilde{v}_{|\Oo}(\pi_G(x))\right|^2\gamma(dx)\\ &=\int_{\Omega_n}\left|\frac{T^G(t+t_0)\widetilde{v}_{|\Oo}(\pi_G(x))-T^G(t_0)\widetilde{v}_{|\Oo}(\pi_G(x))}{t} -L^GT^G(t_0)\widetilde{v}_{|\Oo}(\pi_G(x))\right|^2\gamma(dx)\\
    &=\int_{\Oo}\left|\frac{T^G(t+t_0)\widetilde{v}_{|\Oo}(\xi)-T^G(t_0)\widetilde{v}_{|\Oo}(\xi)}{t} -L^GT^G(t_0)\widetilde{v}_{|\Oo}(\xi)\right|^2\gamma_G(d\xi)
  \end{split}
\end{equation*}
which tends to zero as $t\rightarrow0$. Then $g'(t_0)=L^G T^G(t_0)\widetilde{v}_{|\Oo}(\pi_G(\cdot))$. Moreover $g'$ is continuous with values in $L^2(\Omega_n,\gamma)$ at any $t_0>0$ since
\begin{equation*}
\begin{split}
    \int_{\Omega_n}|g'(t)(x)&-g'(t_0)(x)|^2\gamma(dx)\\
    &=\int_{\Omega_n}|L^GT^G(t)\widetilde{v}_{|\Oo}(\pi_G(x))- L^GT^G(t_0)\widetilde{v}_{|\Oo}(\pi_G(x))|^2\gamma(dx)\\
    &=\int_\Oo|L^GT^G(t)\widetilde{v}_{|\Oo}(\xi)- L^GT^G(t_0)\widetilde{v}_{|\Oo}(\xi)|^2\gamma_G(d\xi)
  \end{split}
\end{equation*}
that goes to zero as $t\rightarrow t_0$. By Lemma \ref{lem:T_G_continuo}, $g\in C((0,\infty);D(L^{(n)}))$ and
\[g'(t)=L^{G}(T^G(t)\widetilde{v}_{|\Oo})(\pi_G(x))=L^{(n)}(T^G(t)\widetilde{v}_{|\Oo}(\pi_G(x)))=L^{(n)}g(t),\ t>0.\]
Moreover
\[g(0)=T^G(0)\widetilde{v}_{|\Oo}(\pi_G(x))=\widetilde{v}_{|\Oo}(\pi_G(x))=f(x)_{|\Omega_n}.\]
Therefore $g$ is a classical solution to
\[\left\{
\begin{split}
  &g'(t)=L^{(n)}g(t),\quad t>0\\
  &g(0)=f_{|\Omega_n},
\end{split}
\right.
\]
in $L^2(\Omega_n,\gamma)$, so that $g(t)=T^{(n)}(t)f_{|\Omega_n}$ and \eqref{eq:identita_Tn_TG} is proved.

Now using Lemma \ref{lem:stima_grad_semigrup_dim_finita} we have:
\[
\begin{split}
  \int_{\Omega_n}\varphi(x)&|\nabla_H T^{(n)}(t)f_{|\Omega_n}(x)|_H\gamma(dx)\\
  &=\int_{\Omega_n}\widetilde{w}(\pi_G(x))|\nabla_H T^{G}(t)(\widetilde{v}_{|\Oo})(\pi_G(x))|_H\gamma(dx)\\
  &=\int_{\Oo}\widetilde{w}(\xi)|\nabla T^{G}(t)(\widetilde{v}_{|\Oo})(\xi)|\gamma_G(d\xi)\\
  &\leq\int_{\Oo}\widetilde{w}(\xi)e^{-t}T^{G}(t)|\nabla (\widetilde{v}_{|\Oo})(\xi)|\gamma_G(d\xi)\\
  &=\int_{\Omega_n}\widetilde{w}(\pi_G(x))e^{-t}T^{G}(t)|\nabla_H(\widetilde{v}_{|\Oo})(\pi_G(x))|_H\gamma(dx)\\
  &=\int_{\Omega_n}\varphi(x) e^{-t}T^{(n)}(t)|\nabla_Hf_{|\Omega_n}(x)|_H\gamma(dx).
\end{split}
\]

Therefore for every $f\in\F C^{\infty}_b(X)$ and $\varphi\in\F C_b(X)$, $\varphi\geq0$ and every $n$ we have
\[\int_{\Omega_n}\varphi(x)|\nabla_H T^{(n)}(t)f_{|\Omega_n}(x)|_H\gamma(dx)\leq\int_{\Omega_n}\varphi(x) e^{-t} T^{(n)}(t)|\nabla_Hf_{|\Omega_n}(x)|_H\gamma(dx).\]
Now thanks to Proposition \ref{prop:limite_Tn_T} we can take the limit as $n\rightarrow\infty$ obtaining
\[\int_{\Omega}\varphi(x)|\nabla_H T^{\Omega}(t)f_{|\Omega}(x)|_H\gamma(dx)\leq\int_{\Omega}\varphi(x) e^{-t} T^{\Omega}(t)|\nabla_Hf_{|\Omega}(x)|_H\gamma(dx)\]
and by Lemma \ref{lem:integrale_scontrato_phi_Cb}
\[|\nabla_H T^{\Omega}(t)f_{|\Omega}(x)|_H\leq e^{-t} T^{\Omega}(t)|\nabla_Hf_{|\Omega}(x)|_H\quad\text{for a.e. }x\in\Omega.\]

If $f\in W^{1,2}(\Omega,\gamma)$ then there exists $\{f_n\}\in\F C^{\infty}_b(X)$ such that ${f_n}_{|\Omega}\rightarrow f$ in $W^{1,2}(\Omega,\gamma)$; along a subsequence $\{f_{n_k}\}$ we have
\[\lim_{k\rightarrow\infty}|\nabla_H T^{\Omega}(t)f_{n_k}(x)|_H= |\nabla_H T^{\Omega}(t)f(x)|_H\quad\text{a.e. in }\Omega\]
and again up to a subsequence we have
\[\lim_{k\rightarrow\infty}T^{\Omega}(t)|\nabla_H f_{n_k}(x)|_H= T^{\Omega}(t)|\nabla_H f(x)|_H\quad\text{a.e. in }\Omega.\]
Therefore
\begin{equation*}
  \begin{split}
    |\nabla_H T^{\Omega}(t)f(x)|_H&=\lim_{k\rightarrow\infty}|\nabla_H T^{\Omega}(t)f_{n_k}(x)|_H \\
    &\leq \lim_{k\rightarrow\infty}e^{-t} T^{\Omega}(t)|\nabla_H f_{n_k}(x)|_H=e^{-t} T^{\Omega}(t)|\nabla_H f(x)|_H
  \end{split}
\end{equation*}
for a.e. $x\in\Omega$.
\end{proof}

\begin{prop}
For all $f$, $g\in L^2(\Omega,\gamma)$ we have
\[\left[T^{\Omega}(t)(fg)\right]^2\leq T^{\Omega}(t)(f^2) T^{\Omega}(t)(g^2)\quad \text{a.e. in }\Omega\]
\end{prop}

\begin{proof}
As before we consider $f,g,\varphi\in\F C_b(X)$ and using Proposition \ref{prop_T(t)(fg)_dim_finita} we prove that
\begin{equation*}
  \begin{split}
    \int_{\Omega_n}\varphi(x)&[T^{(n)}(t)f(x)g(x)]^2\gamma(dx)\\
    &\leq\int_{\Omega_n}\varphi(x)\sqrt{T^{(n)}(t)(f^2)(x)\cdot T^{(n)}(t)(g^2)(x)}\gamma(dx).
  \end{split}
\end{equation*}
The proof is similar to the previous one. Then taking the limit for $n\rightarrow\infty$ we get
\[\int_{\Omega}\varphi(x)[T^{\Omega}(t)f(x)g(x)]^2\gamma(dx)\leq\int_{\Omega}\varphi(x)\sqrt{T^{\Omega}(t)(f^2)(x)\cdot T^{\Omega}(t)(g^2)(x)}\gamma(dx).\]
Since the restrictions to $\Omega$ of the functions of $\F C_b(X)$ are dense in $L^2(\Omega,\gamma)$ we obtain our claim.
\end{proof}

Now we prove the Poincaré inequality. We set
\[m_{\Omega}(\varphi)=\frac{1}{\gamma(\Omega)}\int_\Omega\varphi\ d\gamma,\quad \varphi\in L^1(\Omega,\gamma).\]
We recall that if $\Oo\subset\R^N$ is an open convex set then the Poincaré inequality holds (see Proposition 4.2 in \cite{PratoLun}), that is
\begin{equation}
  \int_\Oo|\psi-m_\Oo(\psi)|^2d\gamma_N\leq\int_\Oo|\nabla\psi|^2d\gamma_N,\quad\forall\psi\in W^{1,2}(\Oo,\gamma_N).
  \label{eq:poincare_ineq_dim_finita}
\end{equation}

\begin{prop}
  For each $\varphi\in W^{1,2}(\Omega,\gamma)$ we have
  \begin{equation}
    \label{eq:poicare_ineq_dim_infinita}
    \int_\Omega|\varphi-m_\Omega(\varphi)|^2d\gamma\leq\int_\Omega|\nabla_H\varphi|_H^2d\gamma
  \end{equation}
\end{prop}

\begin{proof}
First we prove that \eqref{eq:poicare_ineq_dim_infinita} holds for every cylindrical function. Let $f\in\F C^1_b(X)$,
\[f(x)=v(l_1(x),\ldots,l_k(x))\]
 with $v\in C^1_b(\R^k)$ and $l_1,\ldots,l_k\in X^*$. Let $G:=\text{span}\{F_n,R_\gamma(l_1),\ldots,R_\gamma(l_{k})\}$. Then $G$ is a subspace of $H$ of dimension $q\leq n+k$; setting $d=q-\dim F_n$ let $\Oo:=\Oo_n\times\R^d$. Let $\{\widehat{h}_i\}_{i=1}^q$ be an orthonormal basis of $G$ such that $\{\widehat{h}_1,\ldots,\widehat{h}_j\}$ is an orthonormal basis of $F_n$ if $\dim F_n=j$. Then we have
\[
  \begin{split}
    &f(x)=\varphi(\pi_G(x)),\quad\varphi\in C^1_b(G),\\
    &m_{\Omega_n}(f)=\frac{1}{\gamma(\Omega_n)}\int_{\Omega_n}f(x)\ \gamma(dx)=\frac{1}{\gamma_{G}(\Oo)}\int_{\Oo}\varphi(\xi)\ \gamma_G(d\xi)=:m_{\Oo,\gamma_G}(\varphi),\\
    &\lim_{n\rightarrow\infty}m_{\Omega_n}(f)=m_{\Omega}(f).
  \end{split}
\]
Applying \eqref{eq:poincare_ineq_dim_finita} we obtain
\[
 \begin{split}
   \int_{\Omega_n}|f(x)&-m_{\Omega_n}(f)|^2\gamma(dx)\\
   &=\int_{\Omega_n}|f(x)-m_{\Oo,\gamma_G}(\varphi)|^2\gamma(dx) =\int_{\Oo}|\varphi(\xi)-m_{\Oo,\gamma_G}(\varphi)|^2\gamma_G(d\xi)\\
   &\leq\int_{\Oo}|\nabla\varphi(\xi)|^2\gamma_G(d\xi)=\int_{\Omega_n}|\nabla_H f(x)|_H^2\gamma(dx).
 \end{split}
\]

Taking the limit as $n\rightarrow\infty$ we get
\[\int_{\Omega}|f(x)-m_\Omega(f)|^2\gamma(dx)\leq\int_{\Omega}|\nabla_H f(x)|_H^2\gamma(dx).\]
Let now $f\in W^{1,2}(\Omega,\gamma)$. There exists a sequence $\{f_n\}\subset\F C^1_b(X)$ such that ${f_n}_{|\Omega}\rightarrow f$ in $W^{1,2}(\Omega,\gamma)$. It follows that
\begin{equation*}
  \begin{split}
    \int_\Omega|f-m_{\Omega}(f)|^2d\gamma&=\lim_{n\rightarrow\infty}\int_\Omega|f_n-m_{\Omega}(f)|^2d\gamma\\
    &\leq\lim_{n\rightarrow\infty}\int_\Omega|\nabla_H f_n|_H^2d\gamma=\int_\Omega|\nabla_H f|_H^2d\gamma.
  \end{split}
\end{equation*}
\end{proof}

The Poincaré inequality \eqref{eq:poicare_ineq_dim_infinita} implies that if $\nabla_H\varphi$ vanishes in $\Omega$, then $\varphi$ is equal to a constant a.e. in $\Omega$. By the definition of $L^\Omega$, this implies that the kernel of $L^\Omega$ consists of constant functions. As a consequence of the Poincaré inequality other spectral properties of $L^\Omega$ follow.

\begin{prop}
For all $\varphi\in L^2(\Omega,\gamma)$ we have
\begin{equation}
  \|T^{\Omega}(t)\varphi-m_{\Omega}(\varphi)\|_{L^2(\Omega,\gamma)}\leq e^{-t}\|\varphi\|_{L^2(\Omega,\gamma)},\quad t>0,
  \label{eq:decadimento_espon_T(t)}
\end{equation}
and consequently
\begin{equation}
  \sigma(L^\Omega)\backslash\{0\}\subset\{\lambda\in\C: \Re(\lambda)\leq-1\}.
  \label{spectral_gap}
\end{equation}
\end{prop}

\begin{proof}
Let $f\in D(L^{\Omega})$ be such that $m_{\Omega}(f)=0$. By using \eqref{eq:poicare_ineq_dim_infinita} we get
\[\int_\Omega L^{\Omega}f\cdot f d\gamma=-\int_\Omega\langle\nabla_Hf,\nabla_Hf\rangle_Hd\gamma=-\int_\Omega|\nabla_Hf|_H^2d\gamma \leq-\|f\|^2_{L^2(\Omega,\gamma)}.\]
For every $\varphi\in L^2(\Omega,\gamma)$, $m_{\Omega}(\varphi)=0$, $T^{\Omega}(t)\varphi\in D(L^{\Omega})$ for $t>0$. Therefore
\[\frac{d}{dt}\|T^{\Omega}(t)\varphi\|^2_{L^2(\Omega,\gamma)}=2\int_\Omega L^{\Omega}T^{\Omega}(t)\varphi\ T^{\Omega}(t)\varphi\ d\gamma\leq-2\|T^{\Omega}(t)\varphi\|^2_{L^2(\Omega,\gamma)},\quad t>0.\]
It follows that
\begin{equation}
 \label{eq:decadimento_espon_T(t)_per_medie_nulle}
 \|T^{\Omega}(t)\varphi\|^2_{L^2(\Omega,\gamma)}\leq e^{-2t}\|\varphi\|^2_{L^2(\Omega,\gamma)},\quad t>0.
\end{equation}
Let now $\varphi\in L^2(\Omega,\gamma)$. Then
\[
\begin{split}
\int_\Omega|T^{\Omega}(t)\varphi-m_\Omega(\varphi)|^2d\gamma&=\int_{\Omega}|T^{\Omega}(t)(\varphi-m_\Omega(\varphi))|^2d\gamma\leq e^{-2t}\int_{\Omega}|\varphi-m_\Omega(\varphi)|^2\\
&=e^{-2t}\left(\int_{\Omega}|\varphi|^2d\gamma-\gamma(\Omega)[m_\Omega(\varphi)]^2\right)\leq e^{-2t}\int_{\Omega}|\varphi|^2d\gamma.
\end{split}
\]
Note that $\varphi\mapsto m_\Omega(\varphi)$ is the orthogonal projection on $\Ker(L^{\Omega})$. Splitting $L^2(\Omega,\gamma)=\Ker(L^{\Omega})\oplus(\Ker(L^{\Omega}))^{\bot}$, $T^{\Omega}(t)$ maps $(\Ker(L^{\Omega}))^{\bot}=\Ker(m_\Omega)$ into itself and the infinitesimal generator of the restriction of $T^{\Omega}(t)$ to $(\Ker(L^{\Omega}))^{\bot}$ is the part $L_0$ of $L^{\Omega}$ in $(\Ker(L^{\Omega}))^{\bot}$. By \eqref{eq:decadimento_espon_T(t)_per_medie_nulle}, the spectrum of $L_0$ is contained in $\{\lambda\in\C:\ \mathrm{Re}(\lambda)\leq-1\}$. Since the spectrum of $L^{\Omega}$ consists of the spectrum of $L_0$ plus the eigenvalue $0$, \eqref{spectral_gap} follows.
\end{proof}


Now we have all the ingredients to prove the logarithmic Sobolev inequality by using the Deuschel-Strook's method. We write down the proof for the reader's convenience.

\begin{prop}
  For all $f\in W^{1,2}(\Omega,\gamma)$ we have
  \begin{equation}
    \label{eq:log_sobolev_ineq}
    \int_\Omega f^2\log(|f|)d\gamma\leq\int_{\Omega}|\nabla_H f|_H^2d\gamma+\|f\|_{L^2(\Omega,\gamma)}^2\log(\|f\|_{L^2(\Omega,\gamma)}).
  \end{equation}
\end{prop}

\begin{proof}
 First we assume that $f\in\F C^1_b(\Omega)$ and $f\geq c>0$ in $\Omega$. Setting $\varphi=f^2$ we have
 \[\nabla_H f=\frac{1}{2}\frac{\nabla_H\varphi}{\sqrt\varphi}\]
 so \eqref{eq:log_sobolev_ineq} is equivalent to
 \begin{equation*}
 \int_\Omega \varphi\log(\varphi)d\gamma-\int_\Omega\varphi d\gamma\log\left(\int_\Omega\varphi d\gamma\right)\leq\frac{1}{2}\int_{\Omega}\frac{1}{\varphi}|\nabla_H \varphi|_H^2d\gamma.
 \end{equation*}
 We remark that
 \[
 \begin{split}
 \frac{d}{dt}\int_\Omega T^{\Omega}(t)(\varphi)\log(T^{\Omega}(t)(\varphi))d\gamma=&\int_\Omega L^{\Omega}T^{\Omega}(t)(\varphi)\log(T^{\Omega}(t)(\varphi))d\gamma\\
 &+\int_\Omega L^{\Omega}T^{\Omega}(t)(\varphi)d\gamma.
 \end{split}
 \]
 The second term vanishes by the invariance of $\gamma$, while for the first we recall that
 \[\int_\Omega (L^{\Omega}\psi)g(\psi) d\gamma=-\int_\Omega g'(\psi)|\nabla_H\psi|_H^2 d\gamma\]
 with $g(\xi)=\log\xi$ and $\psi=T^{\Omega}(t)\varphi$. Hence
 \[\frac{d}{dt}\int_\Omega T^{\Omega}(t)(\varphi)\log(T^{\Omega}(t)(\varphi))d\gamma=-\int_\Omega\frac{1}{T^{\Omega}(t)(\varphi)}|\nabla_H T^{\Omega}(t)(\varphi)|^2 d\gamma.\]
 Since
 \[(|\nabla_H T^{\Omega}(t)(\varphi)|_H)^2\leq \left(e^{-t}T^{\Omega}(t)|\nabla_H (\varphi)|_H\right)^2\]
 and
 \[\left(T^{\Omega}(t)|\nabla_H (\varphi)|_H\right)^2=\left[T^{\Omega}(t)\left(\sqrt\varphi\frac{|\nabla_H\varphi|_H}{\sqrt\varphi}\right)\right]^2\leq T^{\Omega}(t)(\varphi)T^{\Omega}(t)\left(\frac{|\nabla_H\varphi|^2_H}{\varphi}\right),\]
 we have
 \begin{equation}
 \label{eq:disug_per_log_sob}
 \begin{split}
   \frac{d}{dt}\int_\Omega T^{\Omega}(t)(\varphi)\log(T^{\Omega}(t)(\varphi))d\gamma&\geq-e^{-2t}\int_\Omega T^{\Omega}(t)\left(\frac{|\nabla_H\varphi|^2_H}{\varphi}\right)d\gamma\\
   &=-e^{-2t}\int_\Omega \frac{|\nabla_H\varphi|^2_H}{\varphi}d\gamma.
 \end{split}
 \end{equation}

 We recall that, by \eqref{eq:poicare_ineq_dim_infinita},
 \[\lim_{t\rightarrow\infty} T^{\Omega}(t)\varphi=m_\Omega(\varphi),\quad\text{in }L^2(\Omega,\gamma).\]
 Then
 \[\lim_{t\rightarrow\infty}\int_\Omega|\log(T^{\Omega}(t)\varphi)-\log(m_\Omega\varphi)|^2d\gamma=0\]
 indeed by assumption, we have that $\varphi\geq c^2$ and, by Lemma \ref{lem:positivity_T(t)}, follows that $T^{\Omega}(t)\varphi\geq c^2$. Moreover $m_\Omega(\varphi)\geq c^2$, and
 \[
  \begin{split}
   \left|\log(T^{\Omega}(t)\varphi)-\log(m_\Omega(\varphi))\right|&= \left|\int_{T^{\Omega}(t)\varphi}^{m_\Omega(\varphi)}\frac{1}{s}\ ds\right|\\
   &\leq\frac{1}{\min\{T^{\Omega}(t)\varphi,m_\Omega(\varphi)\}}|T^{\Omega}(t)\varphi-m_\Omega(\varphi)|\\
   &\leq\frac{1}{c^2}|T^{\Omega}(t)\varphi-m_\Omega(\varphi)|.
  \end{split}
 \]
 Then
 \[\lim_{t\rightarrow\infty}\int_\Omega T^{\Omega}(t)(\varphi)\log(T^{\Omega}(t)(\varphi))d\gamma=m_\Omega(\varphi)\log\left(m_\Omega(\varphi)\right)\]
 Integrating \eqref{eq:disug_per_log_sob} with respect to $t$ between $0$ and $\infty$ we get
 \[m_\Omega(\varphi)\log\left(m_\Omega(\varphi)\right)-\int_\Omega\varphi\log(\varphi)d\gamma\geq-\frac{1}{2}\int_\Omega \frac{|\nabla_H\varphi|^2_H}{\varphi}d\gamma.\]
 Now let $f\in W^{1,2}(\Omega,\gamma)$ and $f\geq0$ a.e. in $\Omega$. Then there exists a sequence $\{f_n\}\subset\F C^1_b(\Omega)$ such that $f_n\geq1/n$ and $f_n\rightarrow f$ in $W^{1,2}(\Omega,\gamma)$ and almost everywhere. Since $t^2\log(t)>-1$ for all $t>0$, we can apply the Fatou's lemma and obtain
 \[\begin{split}
 \int_\Omega f^2\log(f)d\gamma&\leq\lim_{n\rightarrow\infty}\int_\Omega f_n^2\log(f_n)d\gamma\\
 &\leq\lim_{n\rightarrow\infty} \left[\int_{\Omega}|\nabla_H f_n|_H^2d\gamma+\|f_n\|_{L^2(\Omega,\gamma)}^2\log(\|f_n\|_{L^2(\Omega,\gamma)})\right]\\
 &=\int_{\Omega}|\nabla_H f|_H^2d\gamma+\|f\|_{L^2(\Omega,\gamma)}^2\log(\|f\|_{L^2(\Omega,\gamma)}).
 \end{split}
 \]
 For a general $f\in W^{1,2}(\Omega,\gamma)$, \eqref{eq:log_sobolev_ineq} follows from the fact that $|f|\in W^{1,2}(\Omega,\gamma)$ and $|\nabla_H |f||_H=|\nabla f|_H$ almost everywhere.
\end{proof}

\begin{rmk}
 The Logarithmic-Sobolev inequality allows to prove an interesting property of the space $W^{1,2}(\Omega,\gamma)$, namely if $f\in W^{1,2}(\Omega,\gamma)$ and $v$ is a measurable function with $|v(x)|\leq k(\|x\|_X+1)$ for some $k>0$ and for a.e. $x\in\Omega$, then $fv\in L^2(\Omega,\gamma)$.
 To this aim we recall that, by the Fernique Theorem \cite{Fernique}, there exists a constant $\alpha>0$ such that
 \[\int_X e^{\alpha\|x\|_X^2}d\gamma<\infty.\]
 Fix $c<\alpha/4$. Then we have
 \[\begin{split}
   \int_{\Omega}&(f(x)v(x))^2 d\gamma\\
   &\leq k^2\int_{\Omega}f(x)^2(\|x\|_X+1)^2 d\gamma\leq 2k^2\int_{\Omega}f(x)^2\|x\|_X^2d\gamma+2k^2\int_{\Omega}f(x)^2d\gamma\\
   &=k^2\int_{\{x\in\Omega:\ c\|x\|_X^2>\log|f(x)|\}}f(x)^2\|x\|_X^2 d\gamma\\
   &\phantom{=}+k^2\int_{\{x\in\Omega:\ c\|x\|_X^2\leq\log|f(x)|\}}f(x)^2\|x\|_X^2 d\gamma+2k^2\|f\|_{L^2(\Omega,\gamma)}^2\\
   &\leq k^2\int_X \|x\|_X^2\ e^{2c\|x\|_X^2}d\gamma+\frac{k^2}{c}\int_\Omega|f(x)|^2\log|f(x)|d\gamma+2k^2\|f\|_{L^2(\Omega,\gamma)}^2\\
   &\leq k^2\left(\int_X \|x\|_X^4 d\gamma\right)^{1/2}\left(\int_Xe^{4c\|x\|_X^2}d\gamma\right)^{1/2}\\
   &\phantom{=}+\frac{k^2}{c}\int_\Omega|f(x)|^2\log|f(x)|d\gamma +2k^2\|f\|_{L^2(\Omega,\gamma)}^2\\
   &\leq C+\frac{k^2}{c}\left(\int_{\Omega}|\nabla_H f|_H^2d\gamma+\|f\|_{L^2(\Omega,\gamma)}^2\log(\|f\|_{L^2(\Omega,\gamma)})\right)+2k^2\|f\|_{L^2(\Omega,\gamma)}^2,
 \end{split}
 \]
 that is $fv\in L^2(\Omega,\gamma)$. The above estimate shows that the functional $\Lambda_v:f\mapsto fv$, maps bounded subsets of $W^{1,2}(\Omega,\gamma)$ into bounded subsets of $L^2(\Omega,\gamma)$; this implies that $\Lambda_v$ is continuous.
\end{rmk}

\begin{ack}
 The author would like to thank Prof. Alessandra Lunardi, Prof. Michele Miranda Jr. for many useful discussions and comments. The author is member of
GNAMPA of the Italian Istituto Nazionale di Alta Matematica (INdAM).
\end{ack}

\end{document}